\documentclass{amsart}

\usepackage{comment, hyperref}
\begin{document}

\newtheorem{theorem}{Theorem}[section]
\newtheorem{prop}[theorem]{Proposition}
\newtheorem{lemma}[theorem]{Lemma}
\newtheorem{cor}[theorem]{Corollary}
\newtheorem{conj}[theorem]{Conjecture}

\theoremstyle{definition}
\newtheorem{definition}[theorem]{Definition}
\newtheorem{rmk}[theorem]{Remark}
\newtheorem{eg}[theorem]{Example}
\newtheorem{qn}[theorem]{Question}
\newtheorem{defn}[theorem]{Definition}

\numberwithin{equation}{section}

\newcommand{\Q}{{\mathbb Q}}
\newcommand\F{{\mathbb F}}
\newcommand{\Ga}{\Gamma}
\newcommand\Z{{\mathbb{Z}}}
\newcommand\natls{{\mathbb{N}}}
\newcommand\ZG{{\mathbb{Z}}G}
\newcommand\til{\widetilde}

\title[$H_1$--semistability for projective groups]{$H_1$--semistability for projective groups}

\author[I. Biswas]{Indranil Biswas}
\address{School of Mathematics, Tata Institute of Fundamental
Research, Homi Bhabha Road, Bombay 400005, India}
\email{indranil@math.tifr.res.in}
\author[M. Mj]{Mahan Mj}
\address{RKM Vivekananda University, Belur Math, WB 711202, 
India}
\email{mahan.mj@gmail.com, mahan@rkmvu.ac.in}



\date{\today}

\thanks{The first author acknowledges the support of the J. C. Bose Fellowship. Research
of the second author is partly supported by CEFIPRA Indo-French Research Grant 4301-1}

\begin{abstract}
We initiate the study of the asymptotic topology of groups that can be realized as fundamental groups
of smooth complex projective varieties with holomorphically convex universal covers
(these are called here as holomorphically convex groups). We prove the $H_1$-semistability
conjecture of Geoghegan for holomorphically convex groups. In view of a theorem
of Eyssidieux, Katzarkov, Pantev and Ramachandran \cite{ekpr}, this implies
that  linear projective groups satisfy the $H_1$-semistability conjecture. 
\end{abstract}

\maketitle

\tableofcontents

\section{Introduction} 

\subsection{Statement of results}

The homological semistability conjecture formulated by Geoghegan, \cite[Conjecture 5,
Section 6.4]{guilb}, is equivalent to the statement that $H^2(G,\, \ZG)$ is free abelian for every
one-ended finitely presented group \cite[Section 13.7]{geogh}, \cite{GM1, GM2}.
(Geoghegan's conjecture was formulated originally as a question in 1979 \cite{guilb}.)
This conjecture has been established (in a stronger form)
for several special classes of groups arising naturally in the context of geometric group
theory: One-relator groups, free products of semistable groups with amalgamation along
infinite groups, extensions of infinite groups by infinite groups, (Gromov) hyperbolic
groups, Coxeter groups, Artin groups and so on \cite{Mi1, Mi2, Mi3, MiTs, MiTs2}. In
this paper we establish this conjecture for groups coming from a completely different
geometric source: fundamental groups of smooth projective varieties with holomorphically 
convex universal cover. For convenience we call them as holomorphically convex groups.

\begin{theorem}[{See Theorem \ref{rest}}]\label{omni1}
Let $G\,=\, \pi_1(X)$ be a torsion-free holomorphically
convex group. Then $H^2(G, \, \ZG)$ is a free abelian group. 
\end{theorem}

A group is called linear if it is a subgroup of $\text{GL}(n, {\mathbb C})$ for
some $n$. If $X$ is a smooth complex projective variety such that $\pi_1(X)$ is
linear, then the universal cover $X$
is holomorphically convex \cite{ekpr}. Therefore, Theorem \ref{omni1} has the
following corollary:

\begin{cor}[{See Corollary \ref{lincd2cor}}]\label{cor-1}
If $G$ is a linear torsion-free projective group,
then $H^2(G, \, \ZG)$ is a free abelian group.
\end{cor}

It is an open question whether the dualizing module of a duality group $G$ is a free
abelian group or not \cite[p. 224]{brown}. It follows from Theorem \ref{omni1} that this
is indeed the case if $G$ is holomorphically convex of cohomological dimension two.

\begin{prop}[{See Proposition \ref{idcor}}]\label{idcor0}
Let $G$ be a holomorphically convex group of dimension two. Then $G$ is a duality
group with free dualizing module.
\end{prop}

The key ingredients in the proofs of Theorem \ref{omni1}, Corollary
\ref{cor-1} and Proposition \ref{idcor0} include
\begin{enumerate}
\item topology (especially second homotopy group) of smooth complex projective 
surfaces with holomorphically convex universal cover (Section \ref{shaf}),

\item a spectral sequence argument for computing group cohomology with local 
coefficients (Section \ref{specs}), which was inspired in part by an argument of Klingler \cite{klingler},

\item homological group theory of duality, inverse duality and Poincar\'e duality groups (Section 
\ref{dpd}), and

\item a theorem of Eyssidieux, Katzarkov, Pantev and Ramachandran \cite{ekpr} 
showing that complex projective manifolds with linear fundamental group have 
holomorphically convex universal cover.
\end{enumerate}

\section{Preliminaries}

\subsection{Holomorphic convexity}

A connected complex manifold $M$ is called {\bf holomorphically 
convex} if for every sequence of points $\{x_i\}_{i=1}^\infty$ of $M$ without 
any accumulation point, there is a holomorphic function $f$ on $M$ such that 
the sequence of nonnegative numbers $\{|f(x_i)|\}_{i=1}^\infty$ is unbounded. 
In this paper, we initiate the study of a natural subclass of projective groups, namely groups that can be realized 
as fundamental groups of smooth complex projective varieties, of dimension
at least two, with holomorphically convex universal covers. We shall
call such groups {\bf holomorphically convex groups}. So any holomorphically
convex group is the fundamental group of a smooth complex projective surface 
with holomorphically convex universal cover (see Proposition \ref{hchid}).
A conjecture of Shafarevich asserts that all smooth
projective varieties have holomorphically convex universal covers.

Let $\mathcal{K}$, $\mathcal{P}$, $\mathcal{HC}$ denote respectively the class of
K\"ahler groups, projective groups and holomorphically
convex groups.

It is clear that $\mathcal{HC} \subset \mathcal{P} \subset \mathcal{K}$. 
The question of reversing the last inclusion is a well-known open problem. 
The following test
question asks explicitly if the first inclusion can be reversed. It can be 
thought of as a group-theoretic version of the Shafarevich conjecture.

\begin{qn} \label{gsc} {\bf (Group-theoretic Shafarevich conjecture)}
Is $\mathcal{HC} = \mathcal{P}$? \end{qn}

The (usual) Shafarevich conjecture certainly implies a positive answer to 
Conjecture \ref{gsc} via the Lefschetz hyperplane Theorem.

\subsection{Semistability}

A subcomplex $A$ of a CW complex $X$ is called \textbf{full} if it is the largest
subcomplex of $X$ among all subcomplexes with the property that the 0-skeleton
coincides with the 0-skeleton $A_0$ of $A$. The full subcomplex of $X$ generated by
the vertices of $X^0 \,\setminus\, A^0$ is called the
CW \textbf{complement} of $A$, and it is denoted by $(X \setminus A)_{CW}$.

An inverse sequence $\{G_n\}_{n\in \mathbb N}$ of groups is called \textbf{semistable} 
if for each $n$ there exists an integer $\phi (n) \,\geq\, n$ such that for all
$k\,\geq\,\phi (n)$, we have $\text{image}(f_n^{\phi (n)})\,= \,\text{image}(f_n^k)$, where
$f^i_j\,:\, G_i\,\longrightarrow\, G_j$ is the homomorphism in the inverse sequence.

A CW complex $X$ satisfies $H_1$--\textbf{semistability} if
the sequence $\{ H_{1} ((\til{X} \setminus K_n)_{CW},\, \Z)\}_{n\in \mathbb N}$ is semistable.
A finitely presented group $G$ satisfies $H_1$--\textbf{semistability} if
some (hence any) finite complex $X$ with $\pi_1(X)\,=\, G$ satisfies $H_1$--semistability. 

It follows from work of Geoghegan and Mihalik \cite[Section 13.7]{geogh}, \cite{GM1, GM2}
that $G$ satisfies $H_1$--semistability if and only if $H^2(G, \,\ZG)$ is free abelian. 

In a certain sense, the notion of semistability was motivated by a Theorem of Farrell
\cite{farrell} which may be thought of as a $H^2(G,\, \ZG)$ (or two-dimensional) version
of the fact that a finitely presented group has 0, 1, 2 or infinitely many ends.
Geoghegan has given examples of {\em finitely generated} 
(but not finitely presented) groups $G$ such
that $H^2 (G, \ZG)$ is not a free abelian group (see \cite[p. 321, Ex. 1]{geogh}).

We shall refer to the integral cohomological dimension of a group $G$ simply as
its dimension. Note therefore that any group of finite dimension is {\it torsion-free}.

\section{Holomorphically convex universal cover}\label{shaf}

A complex analytic manifold $M$ is holomorphically 
convex if and only if it admits a proper holomorphic map
$\Pi : M \longrightarrow Q$ to a Stein space $Q$ such that $\Pi_\ast 
{\mathcal{O}}_M \,=\, {\mathcal{O}}_Q$. The Stein space $Q$ is referred to as 
the Cartan--Remmert reduction of $M$ \cite{rem}.

\begin{prop} Let $G$ be a holomorphically convex group, meaning
$G\,=\, \pi_1(X)$, where $X$ is a smooth complex projective
variety with holomorphically convex universal cover and
$\dim X \, \geq\, 2$. Then there is
a smooth complex projective surface $S$ with holomorphically convex universal 
cover such that $G\,=\, \pi_1(S)$.\label{hchid} \end{prop}

\begin{proof}
Assume that $\dim X \, >\, 2$. Let $\til{X}$ be the universal cover of $X$. Fix 
an embedding $X \, \hookrightarrow\, {\mathbb{CP}}^n$. By the Lefschetz 
hyperplane theorem, the inclusion
$$
\iota\, :\, Y \,:=\, X \cap H\, \hookrightarrow\, X
$$
induces an isomorphism of fundamental groups, where $H$ is a suitable hyperplane. Hence $\iota$ lifts to a proper 
holomorphic embedding
$\til{\iota}\,:\, \til{Y}\,\longrightarrow \,\til{X}$, where $\til{Y}$ is the
universal cover of $Y$. Therefore, $\til Y$ is a complex analytic submanifold
of $\til X$.

Since the embedding of $\til Y$ in $\til X$ is proper,
it follows that if $\{x_n\}$ is a sequence of points of $Y$ without any
accumulation point in $Y$, then $\{x_n\}$ does not have have any
accumulation point in $X$. By holomorphic
convexity of $\til X$, there is a holomorphic function $f$ on $\til X$ such 
that the sequence $\{ |f(x_n)| \}$ is unbounded. Considering the function
$f\circ \iota$ we conclude that $\til Y$ is holomorphically convex.
Now the proposition is deduced inductively.
\end{proof}

\subsection{Higher cohomology groups}

In the rest of this section, we assume $X$ to be a smooth complex projective 
surface
such that the universal cover $\til X$ of $X$ is noncompact 
and holomorphically convex. In particular, $\pi_1(X)$ is an infinite group.
The Cartan--Remmert reduction of $\til X$ will be denoted by $\til Y$. We
note that $\til Y$ is not a point because $\til X$ is noncompact.

Narasimhan, \cite{nar}, and Goresky--MacPherson, \cite{GorMac}, gave 
restrictions on the topology of $\til Y$.

\begin{theorem}[\cite{nar, GorMac}]\label{stein}
Let $M$ be a (not necessarily smooth) connected complex projective surface with 
Stein universal cover $\til{M}$. Then $H^i(\til{M},\, {\mathbb Z})\,=\,0$ for 
$i\,\geq\, 3$. Moreover, $\til{M}$ is homotopy equivalent to a $2$-dimensional 
${\rm CW}$-complex.
\end{theorem}

\subsection{The second homotopy group}

Let $X$ be a smooth complex projective surface, and let $f\, :\, X\, 
\longrightarrow \, X'$ be the minimal model. Then $f_\ast\, :\, \pi_1(X)\,
\longrightarrow \, \pi_1(X')$ is an isomorphism; this is because
$X$ is obtained by successive blow-up of points starting with $X'$.
Therefore, we may, and we will, assume that the surface $X$
under consideration is minimal.

We use the notation $H_i(M)$ (respectively, $H^i(M)$) to denote $H_i(M,\,
\mathbb{Z})$ (respectively, $H^i(M, \,\mathbb{Z})$).

We recall a theorem of Andreotti and Narasimhan \cite{an}.

\begin{theorem}[\cite{an}]\label{free-an} 
Let $M$ be a (not necessarily smooth) complex projective surface 
whose universal cover $\til M$ is Stein. Then the second homotopy 
group $\pi_2(M)$ is free abelian.
\end{theorem}

\begin{prop}\label{free} Let $X$ be a smooth complex 
projective minimal surface whose universal
cover $\til X$ is holomorphically convex. Then the second 
homotopy group $\pi_2(X)$ is free abelian.
\end{prop}

In \cite{gurj}, Gurjar proved that the second homotopy group of 
$X$ is torsion-free if the universal cover of $X$ is 
holomorphically convex. Gurjar has communicated to us an 
observation due to Deligne that the proof in \cite{gurj} leads 
to Proposition \ref{free}. He also showed us his notes on
\cite{gurj}. Since Proposition \ref{free} is not available in 
print, we supply here a proof based on notes of Gurjar on 
\cite{gurj}.

Let
\begin{equation}\label{phi}
\phi \,:\, \til{ X}\,\longrightarrow \,\til{ Y}
\end{equation}
be the Cartan--Remmert 
reduction of $\til{ X}$. There are two cases to consider:

\begin{enumerate}
\item The Cartan--Remmert
reduction $ \til{ Y}$ is a Riemann surface. In this case, 
Proposition \ref{free} has been proved by Gurjar \cite{gurj}.

\item The Stein space $\til{ Y}$ is of complex dimension two. Andreotti and 
Narasimhan, \cite{an}, have proved that in this case
$ H_2 (\til{ X})$ is torsion-free. They further show that if $\til{ X}$ 
has only finitely many compact
analytic subvarieties, then $ H_2 (\til{ X})$ is free abelian.
\end{enumerate}

In the rest of this subsection we assume that the
complex dimension of $\til{ Y}$ is two.

The space $\til Y$ is normal. There is a discrete set $P$ of points of $\til Y$ 
such that
\begin{enumerate}
\item[(a)] $P$ consists of the singularities
of $\til Y$ and some smooth points, and

\item[(b)] $\til X$ is obtained by resolving singularities
of $\til Y$ and blowing up the smooth points in $P$.
\end{enumerate}
Since $X$ is minimal, we know that $X$, hence $\widetilde X$, does not
contain any smooth rational curve that can be contracted. Hence
$P$ consists of singularities alone.

It follows that every non-trivial fiber (meaning the fiber is not a 
single point) $F_0$ of $\phi$ in \eqref{phi} is a finite 
union of irreducible projective curves,
and $H_2 (F_0 )$ is a free abelian group, generated freely by the
homology classes of the irreducible components of $F_0$ \cite{gurj}.
Let $F$ denote the union of all nontrivial fibers of $\phi$. Hence 
$H_2 (F )$ is a free abelian group, generated freely by the
homology classes of all the irreducible components of all 
the non-trivial fibers of $\phi$.
Also $\phi\vert_{\til{X} \setminus F} \,:\, \til{X} \setminus F
\,\longrightarrow \, \til{Y} \setminus P$ is a homeomorphism.

\begin{lemma}\label{inj}
Let $\iota \,:\, F \,\longrightarrow \,\til{X}$ be the inclusion map. 
Then 
$$\iota_\ast\, : \,H_2 (F ) \,\longrightarrow\, H_2 (\til{X} )$$
is an injection. Also $H_3 (\til{X})\,=\,0$.
\end{lemma}

\begin{proof}
First observe that $H_3 (\til{X};\,F) \,=\, H_3 (\til{Y};\,P)$. Indeed,
this follows from the fact that the spaces obtained from $\til{X}$ and 
$\til{Y}$ by coning off $F$ and $P$ respectively, are homotopy equivalent. Since 
$P$ is a discrete set of points, we have
$H_3 (\til{Y}) \,=\, H_3 (\til{Y};\, P)$ by the relative homology exact
sequence for the pair $(\til{Y}\, ,P)$. As
$\til{Y}$ is Stein, we have $H_3 (\til{Y})\,=\,0$ by Theorem \ref{stein}. 
Hence $H_3 (\til{X};\, F) \,=\, 0$. Now
the relative homology exact sequence for the pair $(\til{X}, F)$ 
$$
H_3 (\til{X};\, F)\,\longrightarrow\,H_2 (F )\,\longrightarrow\,H_2
(\til{X} )
$$
yields the first statement of the lemma.

Since $H_3 (\til{X};\, F) \,=\, 0\,=\, H_3 (F )$,
the relative homology exact sequence for the pair 
$(\til{X};\, F)$
$$H_3 (F ) \,\longrightarrow\, H_3 (\til{X} ) 
\,\longrightarrow \,H_3 (\til{X};\, F) $$ yields the second statement.
\end{proof}

\bigskip

\noindent {\bf Proof of Proposition \ref{free}:}\, Let $N 
\,\subset\, \til{Y}$ be a disjoint union of compact contractible neighborhoods
of points in $P$. Let $$\til{N} \,=\, \phi^{-1} (N)\,\subset\,\til{X}\, .$$
Then $\til{N}\cup ( \til{X} \setminus F) \,= \,\til{X}$
and $N \cup ( \til{Y} \setminus P)\,=\,\til{Y}$. We have the
following commutative diagram from the Mayer--Vietoris
sequences for these two spaces:
\begin{equation}
\begin{matrix}
H_3(\widetilde{X}) & \longrightarrow & H_2({\partial N}) & 
\longrightarrow
& H_2(\til{X}\setminus F)\oplus H_2(\widetilde{N}) & \stackrel{i_1}\longrightarrow & 
H_2(\widetilde{X})&\stackrel{j_1}{\longrightarrow}\\
&& ~\Big\downarrow \stackrel{\phi_\ast}\simeq && \quad\, 
\Big\downarrow\stackrel{\phi_\ast}\simeq \quad 
\quad 
\quad \Big\downarrow && ~\Big\downarrow\phi_\ast\\
H_3(\widetilde{Y})\,(=0) & \longrightarrow & H_2(\partial N)& 
\longrightarrow &
H_2(\widetilde{Y} \setminus P) \oplus\,~ \{0\} & \stackrel{i_2}\longrightarrow & 
H_2(\widetilde{Y})&\stackrel{j_2}{\longrightarrow}
\end{matrix}
\end{equation}
$$
\begin{matrix}
&\stackrel{j_1}{\longrightarrow} & H_1(\partial N) & \longrightarrow & H_1(\til{X}\setminus F)\oplus H_1(\widetilde{N})
& \longrightarrow & 0\\
&& ~ \Big\downarrow\stackrel{\phi_\ast}\simeq && ~ \Big\downarrow\stackrel{\phi_\ast}\simeq\,
\,\quad 
\quad \,\,\,\,\, \quad ~\Big\downarrow \\
&\stackrel{j_2}{\longrightarrow} & H_1(\partial N) & \longrightarrow & 
H_1(\widetilde{Y}\setminus P) \oplus\,~ \{0\} & \longrightarrow &
H_1(\widetilde{Y})
\end{matrix}
$$

In the above diagram, $H_3(\widetilde{Y})\,=\,0$ by Theorem \ref{stein}. Also 
$i_1\vert_{H_2(\widetilde{N})}$ is injective by Lemma \ref{inj}. Hence by (a 
slight modification of the proof of) the 5-lemma, the homomorphism $\phi_\ast$ 
induces an isomorphism between $H_2(\widetilde{X})/i_1(H_2(\widetilde{N}))$
and a subgroup of $H_2(\widetilde{Y})$.

Now, $H_2(\widetilde{Y})$ is a free abelian group, because 
$\widetilde{Y}$ is 
Stein (see Theorem \ref{free-an}). Hence the the homomorphism $\phi_\ast$ 
induces an isomorphism between $H_2(\widetilde{X})/i_1(H_2(\widetilde{N}))$
and a free abelian group. It follows that 
$H_2(\widetilde{X})$ is 
isomorphic to $H_2(\widetilde{N}) \oplus H_2(\widetilde{Y})$ as 
$i_1\vert_{H_2(\widetilde{N})}$ is injective. But $H_2(\widetilde{N}) \,=\, 
H_2(F)$ is free. Hence $H_2(\widetilde{X})$ is free. 
 
Finally, since $\widetilde{X}$ is simply connected, it follows from the 
Hurewicz theorem that $\pi_2(X)\,=\,\pi_2(\widetilde{X})\,=\,H_2(\widetilde{X})$ 
is a free abelian group.\hfill $\Box$

\subsection{The second (co)homology groups}

\begin{cor} Let $X$ be a smooth complex projective minimal surface 
whose universal
cover $\til X$ is holomorphically convex. Then the second homology 
group $H_2(\til{X},\, \mathbb{Z})$ is
free abelian. Also, the second cohomology group is isomorphic to 
${\rm Hom}(H_2(\til{X},\, \mathbb{Z}), \Z)$, and it is a direct
product of copies of $\Z$. \label{cohfree}
\end{cor}

\begin{proof} The first statement follows from the Hurewicz theorem as $\til X$ 
is simply connected.

The second statement follows from the universal coefficient theorem. Since 
$H_1(\til{X},\,\mathbb{Z})\,= \,0$, it follows that $H^2(\til{X}, \,
\mathbb{Z})$ is torsion-free. Hence $H^2(\til{X}, \, \mathbb{Z})$ is 
isomorphic to ${\rm Hom}(H_2(\til{X},\, \mathbb{Z}), \Z)$, which is a direct
product of copies of $\Z$, because $H_2(\til{X}, \,
\mathbb{Z})$ is a free abelian group as per the first statement.
\end{proof}

\section{$H_1$--semistability}\label{specs}

\subsection{Second cohomology with local coefficients}

\begin{lemma}\label{rs}
Let $X$ be a smooth complex projective surface with holomorphically
convex universal cover $\til X$. Assume that $G\,:=\, \pi_1(X)$ is 
torsion-free. Let $\phi\,:\,\til{X}\,\longrightarrow\,\til{Y}$ be the 
Cartan--Remmert reduction. Assume that $\til Y$ is an open Riemann 
surface. Then $G$ is isomorphic to the fundamental group of a compact
Riemann surface.
\end{lemma}

\begin{proof}
The Riemann surface $\til Y$ must be biholomorphic to 
either the complex plane or the upper half plane \cite[p. 703]{gurj}. 
Since $G$ is torsion-free, and $\phi$ 
is a proper $G$--equivariant map, it follows that $G$ acts on $Y$ freely 
properly discontinuously and cocompactly by holomorphic automorphisms of 
$\til Y$. Hence $\til{Y}/G$ must be a closed Riemann surface.
\end{proof}

\begin{lemma}\label{isom}
Let $X$ be a smooth complex projective minimal surface with holomorphically 
convex universal cover $\til X$. Let $\phi\,:\, \til{X}\,\longrightarrow\, 
\til{Y}$ be the Cartan--Remmert reduction. Assume that $\til Y$ is a complex 
surface. Also, assume that $G\,:=\,\pi_1(X)$ is torsion-free. Then 
$$\phi^\ast\,:\, H^4_c(\til{ X})\,\longrightarrow\, H^4_c(\til{ Y})$$ is an 
isomorphism.
\end{lemma}

\begin{proof} As before, $P$ denotes the singular locus of $\til{ Y}$, 
which is a discrete subset of points because $\til{ Y}$ is normal. Let 
$$B_1\,\subset\,B_2\,\subset\,\cdots\,\subset\,B_n\,\subset\,\cdots$$ be a 
(relatively) compact exhaustion of $\til Y$ such that $\partial B_n \cap 
P\,=\, \emptyset$ for all $n$. Then $$H^4 (\til Y;\, (\til Y \setminus B_n) \cup 
P)\,=\, H^4 (\til Y;\, (\til Y \setminus B_n))\, ,$$ because $P\cap B_n$ is a 
finite set of points.

Let $\til{B}_n\,:=\,\phi^{-1} (B_n)$. Then $$\til{B}_1\,\subset\,\til{B}_2\, 
\subset\,\cdots\,\subset\,\til{B}_n \,\subset\, \cdots$$ is a (relatively) 
compact exhaustion of $\til X$. Also, if $F\,=\,\phi^{-1} (P)$, then $$H^4 (\til 
X;\, (\til X \setminus \til{B}_n) \cup F)\,=\, H^4 (\til X;\, (\til X \setminus 
\til{B}_n) )\, ,$$ because $F$ is a union of curves (its real 
dimension is two). 

Further, since $\partial B_n \cap P \,=\, \emptyset$, it can be deduced that
$$H^4 (\til X;\, (\til X \setminus \til{B}_n) \cup F)\,=\, H^4 (\til Y;\, (\til 
Y \setminus B_n) \cup P)\, .$$ Indeed, this is
easily seen from the fact that the space obtained by coning off
$(\til X \setminus \til{B}_n) \cup F$ in $\til X$ is homotopy 
equivalent to the space obtained by coning off
$(\til Y \setminus B_n) \cup P$ in $\til Y$. 

Hence $\phi^\ast$ induces an isomorphism between $H^4 (\til X;\, (\til X 
\setminus \til{B}_n)$ and $H^4 (\til Y;\, (\til Y \setminus B_n))$ for 
all $n$. Now the lemma follows by taking limits.
\end{proof}

A Mayer--Vietoris argument gives the following:

\begin{lemma}
Let $X$ be a smooth complex projective minimal surface with holomorphically
convex universal cover $\til X$. Let $\phi\,:\, \til{X}\,\longrightarrow\, 
\til{Y}$ be the Cartan--Remmert reduction. Assume that the complex 
dimension of $\til Y$ is two. Also assume that $G\,:= \,\pi_1(X)$ is 
torsion-free. Define $Y \,:=\, \til{Y}/G$. Then $H^4(Y) 
\,= \,\Z$.\label{h4y}
\end{lemma}

\begin{proof}
Since $\phi$ is equivariant with respect to the actions of 
$G$ on $\til{X}$ and $\til{Y}$, it follows that $Y$ is obtained from $X$
by collapsing some finitely many complex curves (possibly singular) in 
$X$ to points. Let $A_1\, , \cdots\, ,A_n$ be the connected components 
that are thus collapsed to points. Then $Y$ is homotopy equivalent to 
$X$ with cones $cA_i$ attached to $A_i$, $1\leq i\leq n$.
The cohomological Mayer-Vietoris sequence gives $$H^3(\bigcup_{i=1}^n A_i)\, 
\longrightarrow\,H^4(Y)\,\longrightarrow\,H^4(X) \oplus H^4(\bigcup_{i=1}^n 
cA_i) \,\longrightarrow\, H^4(\bigcup_{i=1}^n A_i)\, . $$
Since each $A_i$ is 2-dimensional, and each $cA_i$ is contractible, it follows 
that $H^4(Y)\,=\, H^4(X)\,=\,\Z$.
\end{proof}

\begin{rmk}\label{colln}
Thus we are reduced to looking at smooth complex projective minimal surfaces $X$ 
such that the Cartan--Remmert reduction $\til Y$ of $\til X$ is a complex 
surface. We collect together the topological facts proven so far:

\begin{enumerate}
\item $H^3(\til{X}) \,= \,H^3(\til{Y})\,=\,0$ (Lemma \ref{inj}).

\item $H^4({X}) \,=\, H^4(Y)\,=\,\Z$ (Lemma \ref{h4y}). 

\item $\phi^\ast\,:\, \Z \,=\, H^4_c(\til{ X})\, \longrightarrow\,
H^4_c(\til{ Y})$ is an isomorphism (Lemma \ref{isom}).

\item $\pi_2(X)$, $\pi_2(Y)$, $H_2(\til{X})$, $H_2(\til{Y})$, are free abelian 
groups (Theorem \ref{free-an},
Proposition \ref{free} and Corollary \ref{cohfree}). Also $H^2 (\til{X})$
and $H^2 (\til{Y})$ are direct products of copies of $\Z$ by Corollary \ref{cohfree}.
\end{enumerate}
\end{rmk}

\subsection{Proof of semistability for holomorphically convex groups}

We now set up a Leray--Serre cohomology spectral sequence for the 
classifying maps $$Y \,:=\, \til{Y}/G \,\longrightarrow\, K(G,1)\ ~
\text{ and } ~ \ X\,\longrightarrow\, K(G,1)$$ for the principal $G$--bundles
$\til{X}\,\longrightarrow\, X$ and $\til{Y}\,\longrightarrow\, Y$
respectively. (See \cite[p. 286]{hu}, \cite[Theorem 2.2]{dyer} and 
\cite[Proposition 1]{klingler} for closely related arguments.)

\begin{prop}\label{leraycoh}
Let $X$ be a smooth complex projective minimal surface with 
holomorphically
convex universal cover $\til X$. Assume that $G\,:=\, \pi_1(X)$ is 
torsion-free, and the Cartan--Remmert reduction $\til Y$ of $\til X$
is a complex surface. Let $R$ be any left $\ZG$--module. Then 
$$H^{p+3}(G,\, R)\,=\,H^p(G,\, H^2(\til{M},\, R))$$ for all $p\,\geq\, 3$.

There is an exact sequence of $G$-modules
$$
0 \,\longrightarrow\, H^2(G,\, R) \,\longrightarrow\, H^2(M,\, R)
\,\longrightarrow\, (H^2(\til{M},\, R))^G \,\longrightarrow\,
H^3(G, \,R) \,\longrightarrow\, H^3(M,\, R) \,\longrightarrow 
$$
$$
H^1(G,\, H^2(\til{M},\, R))
\,\longrightarrow\, H^4(G,\, R) \,\longrightarrow\, H^4(M,\, R) 
\,\longrightarrow \, H^2(G,\, H^2(\til{M},\, R))\,
\longrightarrow\, H^5(G,\, R)\, \longrightarrow\, 0\, .
$$
Here $M$ is $X$ or $Y \,:= \,\til{Y}/G$ and $\til{M}$ is
$\til{X}$ or $\til Y$ respectively.
\end{prop}

\begin{proof}
Let $\til{M}\,\longrightarrow \,M$ be a principal $G$--bundle. Take
a $K(G,1)$ space $K$; its universal cover $\widetilde K$ is 
contractible. Let $f\,:\, M\,\longrightarrow\, K$ 
be a classifying map. Let $$\til{g}\,:\, \til{M} \times 
\til{K}\,\longrightarrow\, \til{K}$$
be the natural projection. The group $G$ acts on
$\til K$ through deck transformations, and it acts on
$\til{M} \times \til{K}$ via the diagonal action.
Since $\til g$ is equivariant with respect to these
actions, it induces a map
$$
g\,:\, W\, :=\, (\til{M}\times\til{K})/G\, \longrightarrow \,
{\widetilde K}/G \,=\, K\, .
$$
The fibers of $g$ are homotopy equivalent to $\til M$
(see \cite[pp. 285--286]{hu} for more details).

Note that $H^3(\til{Y})\,= \,0\,= \,H^4(\til{Y})$ by Theorem \ref{stein}. Also,
$\pi_1(\til{Y})\,=\, H_1(\til{Y})\,=\,0$ 
and so $H^1(\til{Y})\,=\,0$ by the universal coefficient theorem. Hence 
$H^i(\til{Y}) \,=\, 0$ for $i \,\neq\, 0, 2$. 

The Leray--Serre cohomology spectral sequence for the above fibration with 
local coefficients~$R$ gives
$$H^p(K,\, (H^q(\til{M} ,\, R)) \,\Longrightarrow \,H^{p+q}(M,\, R)\, ,$$ and 
hence $$H^p(G,\, (H^q(\til{M},\, R))\,\Longrightarrow\, H^{p+q}(M,\, R)$$
since $K$ is a $K(G,1)$ space.

As $H^i(\til{Y}) \,=\, 0$, $i \,\neq\, 0,\, 2$, it follows that $E_2^{p, 0}\,
=\,H^p(G,\, R)$ and $ E_2^{p,2}\,=\, H^p(G,\, H^2(\til{M},\, R))$ are 
the only (possibly) non-zero $E_2^{p, q}$ terms. Since $E_2^{p, 1}\,=\,0$, the 
differential $d_2\,=\,0$. Also, the differentials $d_i$ are zero for 
$i\,>\,3$. Thus $d_3$ is the only (possibly) non-zero differential.
Hence $$E_3\,=\,E_2 ~\, \text{ and }\, ~E_4\,=\,E_5\,=\,\cdots\,= \,
E_{\infty}\, ,$$ and also
\begin{itemize}
\item $E_{\infty}^{0, 0} \,=\, H^0(G,\, H^0(\til{M},\, R))\,=\,
H^0(G,\, R)\,=\,R^G$ (see \cite[p. 58]{brown} for instance), 

\item $E_{\infty}^{1,0}\,=\, H^1(G,\, H^0(\til{M},\, R))\,=\, H^1(G,\,R)$,

\item $E_{\infty}^{2,0}\,=\,H^2(G,\,H^0(\til{M},\, R))\,=\,H^2(G,\,R)$,

\item $E_{\infty}^{p, q}\,=\,H^p(G,\, H^q(\til{M},\, R))\,=\,H^p(G,\, 0)\,
=\,0$, for $q\, \neq\, 0,2$.
\end{itemize}

Further, we have the following exact sequences:
\begin{itemize} 
\item $0\,\longrightarrow\, E_{\infty}^{0,2}\,\longrightarrow\, H^2(\til{M},
\,R)^G \,\stackrel{d_3}{\longrightarrow}\,
H^3(G,\, R)) \,\longrightarrow\, 0$,

\item $0\,\longrightarrow\, H^{p-3}(G,\, H^2(\til{M},\, R))\,
\stackrel{d_3}{\longrightarrow}\, H^p(G,\, R)\,\longrightarrow\,
E_{\infty}^{p, 0}\,\longrightarrow\, 0$, for all $p \,\geq\, 3$, and

\item $0\,\longrightarrow\, E_{\infty}^{p, 2}\,\longrightarrow\, H^p(G,\,
H^2(\til{M},\, R))\,\stackrel{d_3}{\longrightarrow}\,
H^{p+3}(G,\, R)\,\longrightarrow\, 0$, for all $p\,\geq\, 1$.
\end{itemize}

The above descriptions of the $E_{\infty}^{p,q}$ terms can be assembled
to produce the following two exact sequences for the fibration:
$$
0\,\longrightarrow\,H^2(G,\,R)\,\longrightarrow\,H^2(M,\, R)\,\longrightarrow
\,(H^2(\til{M},\, R))^G\,\stackrel{d_3}{\longrightarrow}\, H^3(G,\, R))
$$
(assembling $E_{\infty}^{0,2}$ and $E_{\infty}^{0,2}$), and
$$
H^{p-3}(G,\,H^2(\til{M},\,R))\,\stackrel{d_3}{\longrightarrow}\, H^p(G,\,
R)\,\longrightarrow \,H^p(M,\,R)\,\longrightarrow\, H^{p-2}(G,\,
H^2(\til{M},\, R))
$$
$$
\stackrel{d_3}{\longrightarrow}\, H^{p+1}(G,\, R))~\,\text{ for all }\, p
\,\geq\, 3\, .
$$

Since $H^p(M, R)\,=\,0$ for all $p\,>\,4$, we immediately get from
the above second exact sequence that
$$H^{p+3}(G,\, R)\,=\, H^p(G, \,H^2(\til{M},\, R))$$ for all $ p\,\geq \,3$.
Also, concatenating the first exact sequence with the second
exact sequence for $p\,=\,3,4,5$, we get the long exact sequence in
the proposition.
\end{proof}

Here, as in what follows, we use 
the left $G$--module structure on $\ZG$ to define $H^k(X,\, \ZG)$ and the right
$G$--module structure on $\ZG$ to give its $G$--module structure.

The following is a consequence of Proposition \ref{leraycoh}.

\begin{cor}\label{h2}
Let $X$ be a smooth complex projective minimal surface with holomorphically
convex universal cover $\til X$. Let the Cartan--Remmert reduction $\til Y$ of $\til X$
be a complex surface. Assume that $G \,=\, \pi_1(X)$ 
has dimension less than four. Then $$H^4(M,\, R) \,=\, H^{2}(G,\,H^2(\til{M} , 
\, R))\, .$$ In particular,
\begin{enumerate}
\item \, $\Z \,=\, H^4(M,\, \Z) \,= \,H^{2}(G,\,H^2(\til{M}))$, and
\item \, $\Z \,=\, H^4_c(\til{M},\,\Z) \,=\, H^4(M, \,p_!\Z) \,=
\,H^{2}(G,\,H^2(\til{M},\, p^\ast p_!\Z))$.
\end{enumerate}
Here $M$ is $X$ or $Y \,:= \,\til{Y}/G$ and $\til{M}$ is
$\til{X}$ or $\til Y$ respectively and $p: \til{M} \longrightarrow M$ denotes the (universal) covering map.
\end{cor}

\begin{proof} Lemma \ref{h4y} gives that $H^4(X,\, \Z)\,=\, \Z
\,=\, H^4(Y,\, \Z)$, while Lemma \ref{isom} gives that $H^4(X,\, p_!\Z)\,=\, \Z
\,=\, H^4(Y,\, p_!\Z)$.

The rest follows from the exact sequence in Proposition \ref{leraycoh},
putting $H^4(G,\, R)\,=\,H^5(G,\, R)\,=\,0$.
\end{proof}

We are now in a position to prove $H_1-$semistability for holomorphically
convex groups. 

\begin{theorem}
Let $G\,=\, \pi_1(X)$ be a torsion-free group that is the fundamental group of a 
smooth complex projective variety $X$ with holomorphically
convex universal cover. Then there exists an exact sequence of $G$-modules
$$0\longrightarrow H^2(G,\, \ZG)\,\longrightarrow\, \pi_2(X)\,\longrightarrow\, 
{\rm Hom}_G(\pi_2(X),\, \ZG)\,\longrightarrow\, H^3(G,\, \ZG)\,\longrightarrow\,
0\, .$$ It follows that $H^2(G, \, \ZG)$ is a free abelian group. \label{rest}
\end{theorem}

\begin{proof}
By the Lefschetz hyperplane Theorem, we can assume, without loss of generality that $X$
is a surface. From Proposition \ref{leraycoh}, there is an exact sequence of $G$-modules,
$$
0 \,\longrightarrow\, H^2(G,\, \ZG) \,\longrightarrow\, H^2(X,\, p_!\Z)
\,\longrightarrow\, (H^2(\til{X},\, p^\ast p_!\Z))^G \,\longrightarrow\,
H^3(G, \,\ZG) \,\longrightarrow\, H^3(X,\, p_!\Z)\, .
$$

The first statement of the proposition follows from the following sequence of 
observations.\\
1) We have $H^2(X,\, p_!\Z)\,=\,H^2_c(\til{X},\, \Z)\,=\,H_2(\til{X},\, \Z) 
\,=\, \pi_2(X),$
where the first equality is the standard interpretation for cohomology with $p^\ast p_!\Z$ (isomorphic to $\ZG$) coefficients (see \cite[p. 209]{brown} for instance),
the second equality follows from Poincar\'e duality applied to $\til X$, and 
the third equality follows from the Hurewicz Theorem.\\
2) We have $H^3(X,\, p_!\Z) \,=\, H^3_c(\til{X},\, \Z)\,=\,H_1(\til{X},\, 
\Z)\,= \, 0$, by a similar argument.\\
3) Next, $H^2(\til{X},\, p^\ast p_!\Z) \,= \, {\rm Hom}(H_2(\til{X}),\, \ZG)\,=\, 
{\rm Hom}(\pi_2(X),\, \ZG)$.\\
4) Finally, ${\rm Hom}(M,\,N)^G \,=\, {\rm Hom}_G(M,\,N).$

In particular, $H^2(G,\, \ZG)$ injects into $\pi_2(X)$ which is free abelian by 
Proposition \ref{free}. The second statement of the proposition follows.
\end{proof}

As a consequence of Theorem \ref{rest} we deduce $H_1$--semistability as
it is usually defined (cf. \cite{geogh}). Theorem 13.3.3 of \cite{geogh} combined with
Theorem \ref{rest} gives:

\begin{cor} \label{ss} Let $G\,=\, \pi_1(X)$ be a torsion-free group that is the
fundamental group of a smooth complex projective surface $X$ with holomorphically
convex universal cover. Let $K_i$ be a (filtered) CW exhaustion of $\til X$. Then 
the sequence $\{ H_{1} ((\til{X} \setminus K_i)_{CW},\, \Z)\}_{i\geq 0}$ is semistable.
\end{cor}

The following theorem of Eyssidieux, Katzarkov,
Pantev and Ramachandran \cite{ekpr} (see also \cite{katram, eyssi}) shall be 
needed in obtaining a useful Corollary:

\begin{theorem}[\cite{ekpr}]\label{linsc}
Let $X$ be a smooth complex projective variety such that its fundamental
group $\pi_1(X)$ is linear (meaning a subgroup of ${\rm GL}(n,{\mathbb C})$
for some $n$). Then the universal cover $\til X$ of $X$ 
is holomorphically convex.
\end{theorem}

Theorem \ref{rest} and Theorem \ref{linsc} together immediately give the 
following: 

\begin{cor}\label{lincd2cor}
Let $G$ be a linear torsion-free projective group. Then $H^2(G, 
\, \ZG)$ is a free abelian group.
\end{cor}

\subsection{Duality groups}\label{dpd}

We refer to \cite{be} (see also \cite[p. 220]{brown}) for details on duality groups. In this Section, we deduce (as a consequence of
Theorem \ref{rest}) that the dualizing module for a 2-dimensional holomorphically convex group is free abelian as a group.
It is unknown \cite{brown}[p. 224] if this is always the case for duality groups.

\begin{theorem}[Bieri--Eckmann] \label{dual}
A group $G$ is a duality group of dimension $n$ if it satisfies one
the following two equivalent conditions:
\begin{itemize}
\item There exists a $\ZG$-module $I$ (called the {\bf dualizing module}) such 
that for any $\ZG$-module $A$, there is an isomorphism induced by cap-product
with a fundamental class: $$H^i(G,\, A) \,\simeq\, H_{n-i}(G,\, I
\otimes_\Z A)\, .$$
\item $G$ is of type $FP$ and 
$$H^i(G, \,\ZG) \,= \begin{cases}
0, \quad \text{for} \; i \not=n\, , \\
I, \quad \text{for} \; i =n\, .
\end{cases}
$$
\end{itemize}
If the equivalent conditions hold, then $I$ is isomorphic to $H^n(G,\, \ZG)$ as 
a $\ZG$-module and is torsion-free as an abelian group.
\end{theorem}

As a consequence of Theorem \ref{rest} and Theorem \ref{dual} we have
the following:

\begin{prop}\label{idcor}
Let $G$ be a holomorphically convex group of dimension two. Then $G$ is a duality
group with free dualizing module.
\end{prop}

We end with the following simple fact about duality groups.

\begin{prop}\label{homal} Let $G$ be a two-dimensional duality group, and let 
$I$ be the dualizing module.
Then for any $\ZG$-module $N$, the cohomology $H^{2}(G,\, N \otimes \ZG)$ is
isomorphic to $N \otimes I$.
\end{prop}

\begin{proof}
Since $G$ is a two-dimensional duality group, the dualizing module $I$ is
isomorphic to $H^{2}(G,\, \ZG)$ as a $\ZG$--module. Hence, by Theorem 
\ref{dual},
for any $G$-module $Q$, we have $H^{2}(G,\,Q)\,=\,H_0(G,\, Q \otimes I)$.
Further, $H_0(G,\, Q \otimes I)\,=\,(Q \otimes I)_G$ (see \cite[p. 55]{brown}).

Next, taking $Q\,=\, N \otimes \ZG$ it follows that $H^{2}(G,\,N \otimes \ZG)$ 
is isomorphic to $(N\otimes \ZG \otimes I)_G$, which, in turn,
is isomorphic to $(N \otimes I \otimes \ZG )_G$.

Finally, for any $\ZG$--modules $B$ and $C$, we have $(B \otimes C)_G \,=\,
B\otimes_{\Z G} C$ \cite[p. 55]{brown}. Hence $$(N \otimes I \otimes \ZG 
)_G\,=\,((N
\otimes I) \otimes \ZG )_G\,= \,(N \otimes I) \otimes_{\Z G} \ZG\,=\,N \otimes
I\, .$$
Thus $H^{2}(G,\, N \otimes \ZG)$ is isomorphic to $N\otimes I$.
\end{proof}

\section*{Acknowledgments}

We thank R. V. Gurjar for generously sharing his notes on \cite{gurj} with us 
and for telling us that Deligne has observed that the techniques of \cite{gurj} 
lead to a proof of Proposition \ref{free}. We are very grateful to Gadde Anandaswarup
for help with homological algebra and to Ross Geoghegan for helpful email-correspondence
on semistability. Different parts of this work were done while the authors were visiting
Harish-Chandra Research Institute, Allahabad, and Institute of Mathematical Sciences,
Chennai and during a visit of the second author to Tata Institute of Fundamental
Research, Mumbai. We thank these institutions for their hospitality.

\end{document}